\documentclass[12pt]{amsart}

\makeatletter

\@addtoreset{equation}{section}
\makeatother
\usepackage[margin=3cm]{geometry}
\newenvironment{nouppercase}{%
  \renewcommand{\uppercasenonmath}[1]{}}{}
\usepackage{amsmath,amssymb,amsthm,color}
\usepackage[T1]{fontenc}
\DeclareMathOperator{\Ima}{Im}
\theoremstyle{definition}
\newtheorem{thm}[equation]{Theorem}
\newtheorem{prop}[equation]{Proposition} 
 
\newtheorem*{conj*}{Conjecture}
\newtheorem*{thm*}{Theorem}
\newtheorem{rem}[equation]{Remark} 
\newtheorem{cor}[equation]{Corollary}

\newtheorem{lem}[equation]{Lemma}
\newtheorem{exam}[equation]{Example}
\begin{document}

\title{On the Bhargava factorial of polynomial maps}
\author{Wataru Takeda}
\address{Department of Applied Mathematics, Tokyo University of Science,
1-3 Kagurazaka, Shinjuku-ku, Tokyo 162-8601, Japan.}
\email{w.takeda@rs.tus.ac.jp}

\begin{abstract}
Bhargava introduced a generalization of the factorial function to
extend classical results in integers to Dedekind rings and unify them.
We study the Bhargava factorial of the images of polynomial maps from an analytic perspective.
We first give the $\mathfrak p$-adic closures of the images of polynomial maps, which is the key to compute $\mathfrak p$-adic part of the Bhargava factorial. Then, as a special case, we give the Stirling formula for the image of quadratic polynomials with integer coefficients. 
\end{abstract}
\begin{nouppercase}
    \maketitle
\end{nouppercase}

\section{Introduction}
Bhargava generalized the factorial function $l!$ to
generalize classical results in $\mathbb Z$ to Dedekind rings and unify them \cite{bh97,bh00}. 
Let $S$ be an infinite subset of a Dedekind ring $R$. 
First, we define a $\mathfrak p$-ordering of $S$ for a prime ideal $\mathfrak p$ of $R$. A $\mathfrak p$-ordering of $S$ is any sequence $(a_n)_{n=0}^\infty$ of elements of $S$ that is formed as follows:
\begin{itemize}
    \item We choose any element $a_0\in S$;
    \item For $k\ge1$, we choose an element $a_n\in S$ such that 
\[v_{\mathfrak p}^R\left(\prod_{k=0}^{n-1}(a_n-a_k)\right)= \inf_{x\in S}v^R_{\mathfrak p}\left(\prod_{k=0}^{n-1}(x-a_k)\right),\]
where $v_{\mathfrak p}^R$ is the $\mathfrak p$-adic valuation defined by 
$v_{\mathfrak p}^R(a)=\max\{v\ge0~|~\mathfrak p^v|(a)R\}$.
\end{itemize}
For a $\mathfrak p$-ordering of $S$, we construct the $\mathfrak p$-sequence $(v_{\mathfrak p}(n;S))_{n=1}^\infty$ as 
\[v^R_{\mathfrak p}(n;S)=v^R_{\mathfrak p}\left(\prod_{k=0}^{n-1}(a_n-a_k)\right).\]  
It is known that the associated $\mathfrak p$-sequence of $S$ is independent of the choice of $\mathfrak p$-ordering of $S$. 
With these settings, we define the Bhargava factorial $l!_S^R$ by 
\[l!_S^R=\prod_{\mathfrak p:\text{prime ideal}}\mathfrak p^{v_{\mathfrak p}^R(l;S)}.\]
Let $R_{\mathfrak p}$ be the localization of $R$ at $\mathfrak p$. Then $R_{\mathfrak p}$ is a discrete valuation ring with the maximum ideal $\mathfrak m=\mathfrak pR_{\mathfrak p}$ and valuation $v_{\mathfrak m}^{R_{\mathfrak p}}$. We denote by $\tau:R\rightarrow R_{\mathfrak p}$ the standard embedding of $R$ into $R_{\mathfrak p}$.
By the definition, one can confirm that 
\begin{equation}
    \label{loc}
    {v_{\mathfrak m}^{R_{\mathfrak p}}(l;\tau(S))}={v_{\mathfrak p}^R(l;S)}.
\end{equation}
In this paper, if there is no confusion, we use the abbreviation $v_{\mathfrak p}(l;S)$ and $l!_S$ for $v_{\mathfrak p}^R(l;S)$ and $l!_S^R$, respectively.

We now give some examples of the Bhargava factorial. 
When $S=R=\mathbb Z$, we can choose the natural ordering $0,1,2,3,\ldots$ as $p$-ordering of $l!_S$ for all primes $p$ and find $l!_S$ is the ordinary factorial $l!$. 
Also, when $S(a,b)=\{an+b~|~n\in\mathbb Z\}$ for some $a,b\in\mathbb Z$ then $l!_{S(a,b)}=a^ll!$. 

Bhargava proposed a generalization for several variables in \cite{bh00} and pointed out that analogs of properties of the Bhargava factorial also hold for this generalization. Evrard \cite{ev12} extended the Bhargava factorial in several variables and provided proofs of all Bhargava’s claims in \cite{bh00}.
Several years later, Rajkumar, Reddy, and Semwal also presented a generalization of this factorial \cite{kad18}. 
These papers generalized classical algebraic properties in $\mathbb Z$ to arbitrary subsets $\underline{S}\subset R^n$.

On the other hand, Bhargava also introduced many analytic questions in \cite{bh00}.
For example, Question 31 ``What are analogs of Stirling's formula for generalized factorials?''.
Recently, Daiz obtained asymptotic formulas for the set of primes $\mathbb P$ and a certain set $S$ associated with the sequence found as A202367 at www.oeis.org. \cite{Da20}.
More precisely, it is shown that 
\[\log(l + 1)!_{\mathbb P} = \log l! + Cl + o(l),\]
where
\[C =\sum_{p\in\mathbb P}\frac{\log p}{
(p-1)^2}= 1.2269688\cdots\]
and similar form formula for $l!_{S}$.
In this paper, we also study the analytic properties of the Bhargava factorials of the images of polynomial maps $f(R)$ and present an analytic estimate of a certain Bhargava factorial for subsets of $\mathbb Z$.
This gives a partial solution to Bhargava's question \cite{bh00}.

An outline for the rest of this article is as follows.
In Section \ref{secondsec}, we review that the $\mathfrak p$-adic part of Bhargava factorials of $S$ is depended on only the $\mathfrak p$-adic closures ${\overline S}^{\mathfrak p}$ and prove a weak form of Hensel's lemma to identify the $\mathfrak p$-adic closure of $f(R)$.
Thereafter, in Section \ref{thirdsection}, we focus on quadratic polynomial $f(x)\in R[x]$ under the assumption that $R$ is a principal ideal domain. By \eqref{loc}, this assumption does not restrict the case. In Theorem \ref{quad}, which is one of the main results of this paper, we give the $\mathfrak p$-adic closure of $f(R)$ all but one case.
In Section \ref{stirlingsec}, we deal with the case $R=\mathbb Z$. When $f$ is quadratic, we can characterize $\mathfrak p$-adic closure of $f(\mathbb Z)$ for any $\mathfrak p$. As one of its corollaries, we give a kind of the Stirling formula for Bhargava factorial.


\section{Estimate of $\mathfrak p$-adic part of Bhargava factorials}
\label{secondsec}
In this section, we prove key results for calculating $v_{\mathfrak p}(l;f(R))$ with $f(x)\in R[x]$. 
The following fact verifies that it suffices to consider the $\mathfrak p$-adic closures ${\overline S}^{\mathfrak p}$ of $S$ in $R$.
\begin{lem}[\cite{ac15}]
\label{clo}
Let $S \subset T$ be two subsets of a Dedekind ring $R$
and let $\mathfrak p$ be a prime ideal of $R$. Then, the following assertions are equivalent:
\begin{enumerate}
\item For all $l \in \mathbb Z_{\ge0}$, $v_{\mathfrak p}(l;S) = v_{\mathfrak p}(l;T)$,
\item the $\mathfrak p$-adic closures ${\overline S}^{\mathfrak p}$ and ${\overline T}^{\mathfrak p}$ of $S$ and $T$ in $R$ are equal.
\end{enumerate}
\end{lem}
In this paper, we focus on the Bhargava factorials of $f(R)$. To deal with this, we introduce the following lemma which takes advantage of polynomials' strengths.
The following lemma is a weak form of Hensel's lemma. (See \cite{ma} for the detail of the original Hensel's lemma).
\begin{lem}[Weak form of Hensel's lemma]
\label{cha}
Let $R$ be a principal ideal domain. For fixed a prime ideal $\mathfrak p=(\pi)$ of $R$, we assume the residue field $k=R/\mathfrak p$ is finite. 
We denote by $C_{\mathfrak p}$ a complete system of distinct representatives for $R/\mathfrak p$. 
Then for any polynomial $f(x)\in R[x]$ and $a\in C_{\mathfrak p}$, the following first assertion implies the following second assertion. 
\begin{enumerate}
    \item There exists $b\in R$ such that $f(b)\equiv a\mod \mathfrak p$ and $f'(b)\not\in\mathfrak p$.
    \item For any positive integer $k$ and for any $c_i\in C_{\mathfrak p}$, there exists $b\in R$ such that $f(b)\equiv a+c_1\pi+\cdots+c_{k-1}\pi^{k-1}\mod \mathfrak p^k$.
\end{enumerate}
\end{lem}
\begin{proof}
Let $f(x)=\sum_{i=0}^na_ix^i$. By considering the number of the representatives $c\in C_{\mathfrak p}$ for $R/\mathfrak p^k$ with $c\equiv a\mod \mathfrak p$, it suffices to show that if $f(b+\beta_1\pi) \equiv f(b+\beta_2\pi) \mod \mathfrak p^k$ then $\beta_1\equiv\beta_2\mod \mathfrak p^{k-1}$. Subtracting $f(b+\beta_1\pi)-f(b+\beta_2\pi)$, we have
\begin{align*}
 f(b+\beta_1\pi)-f(b+\beta_2\pi)&=(\beta_1-\beta_2)\sum_{i=1}^n\frac{f^{(i)}(b)}{i!}\frac{\beta_1^i-\beta_2^i}{\beta_1-\beta_2}\pi^i\\
 &=(\beta_1-\beta_2)\left(\sum_{i=2}^n\frac{f^{(i)}(b)}{i!}\frac{\beta_1^i-\beta_2^i}{\beta_1-\beta_2}\pi^i+f'(b)\pi\right).
\end{align*}
By the assumption $f'(b)\not\in \mathfrak p$, the second factor of the last formula 
\[\sum_{i=2}^n\frac{f^{(i)}(b)}{i!}\frac{\beta_1^i-\beta_2^i}{\beta_1-\beta_2}\pi^i+f'(b)\pi\not\equiv0\mod \mathfrak p^2.\] 
Therefore, if $f(b+\beta_1\pi) \equiv f(b+\beta_2\pi) \mod \mathfrak p^k$ then $\beta_1\equiv\beta_2\mod \mathfrak p^{k-1}$. This is the desired conclusion.
\end{proof}
For a fixed complete system of representatives $C_{\mathfrak p}$ for $R/\mathfrak p$, we define two subsets of $C_{\mathfrak p}$ as \[S_{f,\mathfrak p}=\{a\in C_{\mathfrak p}~|~\exists b, f(b)\equiv a\mod\mathfrak p, f'(b)\not\equiv 0\mod\mathfrak p\}\] and \[S_{f,\mathfrak p}^-=\{a\in C_{\mathfrak p}~|~\exists b, f(b)\equiv a\mod\mathfrak p,f'(b)\equiv 0\mod\mathfrak p\}.\]
For example, if $f(x)=x^2$ then $S_{f,\mathfrak p}$ is the set of quadratic residues modulo $\mathfrak p$ excluding $0$.
\begin{cor}
\label{bit}
Let $\mathfrak p=(\pi)$ be a prime ideal of a principal ideal domain $R$ and let $f(x)\in R[x]$. Then 
\[\overline{f(R)}^{\mathfrak p}=\bigsqcup_{a\in S_{f,\mathfrak p}}(a+\mathfrak p)\sqcup E,\]
where $E$ is caused from $S_{f,\mathfrak p}^-\setminus S_{f,\mathfrak p}$ and satisfies  \[E\subset\bigsqcup_{a\in S_{f,\mathfrak p}^-\setminus S_{f,\mathfrak p}}\bigcup_{\substack{b\in C_{\mathfrak p}\\f(b)\equiv a\mod\mathfrak p}}(f(b)+\mathfrak p^{2}).\]
\end{cor}
\begin{proof}
First, we deal with elements $\alpha\in a+\mathfrak p$ with $a\in S_{f,\mathfrak p}$. By Lemma \ref{cha}, for all $k\in\mathbb Z_{\ge0}$ there exists $b\in R$ such that $\alpha\equiv f(b)\mod \mathfrak p^k$. Therefore, $\alpha\in \overline{f(R)}^{\mathfrak p}$. 

Next, we consider elements in $a+\mathfrak p$ with $a\in S_{f,\mathfrak p}^-\setminus S_{f,\mathfrak p}$. Let $b\in R$ such that $f(b)\equiv a\mod \mathfrak p$. Since $f'(b)\in\mathfrak p$, for all $\beta=b+\beta_1\pi\in b+\mathfrak p$, we have 
\begin{align*}
    f(\beta)-f(b)&=f(b+\beta_1\pi)-f(b)\\
    &=\sum_{i=2}^n\frac{f^{(i)}(b)}{i!}\beta_1^{i}\pi^i+f'(b)\beta_1\pi\\
    &\equiv 0\mod\mathfrak p^2.
\end{align*}
Therefore, for all $\beta\in b+\mathfrak p$, we have $f(b)\equiv f(\beta)\mod\mathfrak p^2$ and $f(\beta)\in f(b)+\mathfrak p^2$. 
Applying the same argument to other $a\in S_{f,\mathfrak p}^-\setminus S_{f,\mathfrak p}$ and $b\in C_{\mathfrak p}$ with $f(b)\equiv a\mod\mathfrak p$, we obtain 
 \[E\subset\bigsqcup_{a\in S_{f,\mathfrak p}^-\setminus S_{f,\mathfrak p}}\bigcup_{\substack{b\in C_{\mathfrak p}\\f(b)\equiv a\mod\mathfrak p}}(f(b)+\mathfrak p^{2}).\]
This proves this corollary.
\end{proof}
If $S_{f,\mathfrak p}^-\setminus S_{f,\mathfrak p}=\emptyset$, then we can avoid dealing the uncertain set $E$. For example, when we consider a polynomial $f(x)=x^3-x^2\in \mathbb Z[x]$ and $p=5$, we find that the solution to $f'(x)\equiv0\mod 5$ is $x\equiv0,4\mod5$.
On the other hand, we confirm that $f(0)=f(1)$ and $f(3)\equiv f(4)\mod5$. 
Thus, it suffices to compute $\overline{f(5\mathbb Z+1)}^{5}$ and $\overline{f(5\mathbb Z+3)}^{5}$ instead of $\overline{f(5\mathbb Z)}^{5}$ and $\overline{f(5\mathbb Z+4)}^{5}$. 
We note that $\overline{f(5\mathbb Z+4)}^{5}\subset\overline{f(5\mathbb Z+3)}^{5}$ and $\overline{f(5\mathbb Z)}^{5}\subset\overline{f(5\mathbb Z+1)}^{5}$.
More generally, we present some examples for $S_{f,\mathfrak p}^-\setminus S_{f,\mathfrak p}=\emptyset$ in the following.
\begin{exam}
\label{24}
Let $R$ be a Dedekind ring with ${\rm char}(R)\neq3$ and $f(x)=x^3-x^2\in R[x]$. Then, it holds that for any prime ideal $\mathfrak p$ of $R$,
\[\overline{f(R)}^{\mathfrak p}=\bigsqcup_{a\in S_{f,\mathfrak p}}(a+\mathfrak p).\]
\end{exam}
\begin{proof}
It suffices to prove that for any $a\in S_{f,\mathfrak p}^-$ there exists an element $b\in S_{f,\mathfrak p}$ with $f(a)\equiv f(b)\mod \mathfrak p$. 

First, we consider the case ${\rm char}(R)=2$. In this case, the equation $f'(x)=0$ has the only solution $x=0$. As $f(0)=f(1)$, we have $S_{f,\mathfrak p}^-\setminus S_{f,\mathfrak p}=\emptyset$ for each prime ideal $\mathfrak p$.

Next, we deal with the case ${\rm char}(R)\neq2,3$. If $v_{\mathfrak p}(6)\ge1$ then $x\equiv0\mod \mathfrak p$ is the only solution to $f'(x)\equiv0\mod \mathfrak p$. 
As above, the identity $f(0)=f(1)$ ensures that $S_{f,\mathfrak p}^-\setminus S_{f,\mathfrak p}=\emptyset$.
For any prime ideal $\mathfrak p$ with $v_{\mathfrak p}(6)=0$, we have two solutions $x\equiv 0, 2\cdot 3^{-1}\mod \mathfrak p$ to $f'(x)\equiv0\mod \mathfrak p$, where $3^{-1}$ is a preimage of an inverse of $3$ in mod $\mathfrak p$. As $f(0)\equiv f(1)\mod\mathfrak p$ and $f(2\cdot 3^{-1})\equiv f(-1\cdot 3^{-1})\mod\mathfrak p$, we also obtain $S_{f,\mathfrak p}^-\setminus S_{f,\mathfrak p}=\emptyset$. Thus, the assertion holds.
\end{proof}
More generally, we may show the following result in a similar argument.
\begin{exam}
Let $R$ be a Dedekind ring with ${\rm char}(R)\neq3$. Let $\alpha,\beta\in R$ and let $f(x)=x^3-(2\alpha+\beta)x^2+(\alpha^2+2\alpha\beta)x\in R[x]$. Then, we find that
for any prime ideal $\mathfrak p$ of $R$ with $\alpha\not\equiv \beta\mod \mathfrak p$,
\[\overline{f(R)}^{\mathfrak p}=\bigsqcup_{a\in S_{f,\mathfrak p}}(a+\mathfrak p).\]
\end{exam}
\begin{proof}
As in the proof of Example \ref{24}, it suffices to prove that for any $a\in S_{f,\mathfrak p}^-$ there exists an element $b\in S_{f,\mathfrak p}$ with $f(a)\equiv f(b)\mod \mathfrak p$.

Since $f'(x)=(3x-(\alpha+2\beta))(x-\alpha)$, 
the candidate elements in $S_{f,\mathfrak p}^-$ are $x\equiv\alpha,(\alpha+2\beta)\cdot 3^{-1}\mod \mathfrak p$, if $3$ is invertible mod $\mathfrak p$. We may find that $f(\alpha)= f(\beta)$ and $f((\alpha+2\beta)\cdot 3^{-1})\equiv f((4\alpha-\beta)\cdot 3^{-1})\mod \mathfrak p$. The assumption $\alpha\not\equiv \beta\mod \mathfrak p$ ensure that $f'(\beta)\not\equiv 0\mod \mathfrak p$ and $f'((4\alpha-\beta)\cdot 3^{-1})\not\equiv 0\mod \mathfrak p$.
Thus, we can apply the same argument as the proof of Example \ref{24} and obtain the desired assertion.
\end{proof}

\section{Quadratic polynomials}
\label{thirdsection}
In this section, we assume $f$ is a quadratic polynomial. As we remark later, the higher degree cases become much more complicated than the quadratic case.

By (\ref{loc}) and Lemma \ref{clo}, it suffices to consider the $\mathfrak pR_{\mathfrak p}$-adic closure of $(\tau\circ f)(R_{\mathfrak p})$ in $R_{\mathfrak p}$ for each prime ideal $\mathfrak p$ of $R$. Therefore, in this section, we assume $R$ is a principal ideal domain and $\mathfrak p=(\pi)$.

Let $f(x)=a_2x^2+a_1x+a_0\in R[x]$ and let $v(\mathfrak p,f)=\min\{v_{\mathfrak p}(a_1),v_{\mathfrak p}(a_2)\}$. Then we define the important ideal:
\[\mathfrak g_f=\prod_{\mathfrak p}\mathfrak p^{v(\mathfrak p,f)}.\] By the definition of the Bhargava factorial, we find that \[l!_{f(R)}=\mathfrak g_f^l\times \prod_{\mathfrak p}l!_{f(R,\mathfrak p)},\]
where $f(R,\mathfrak p)$ is the following subset of $R$:
\[\left\{\frac{f(x)-a_0}{\pi^{v(\mathfrak p,f)}}~|~x\in R\right\}.\]
In the following, we assume $\mathfrak g_f=(1)$. 
We denote by $C_{\mathfrak p}=\{0,c_2,\ldots,c_N\}$ complete system of distinct representatives for $R/\mathfrak p$.
Then, we prove the following theorem.
\begin{thm}
\label{quad}
Let $R$ be a principal ideal domain and $\mathfrak p$ be a prime ideal of $R$. We assume that the residue class field $\kappa=R/\mathfrak p$ is finite.
For quadratic polynomial $f(x)=a_2x^2+a_1x+a_0\in R[x]$ with the ideal $\mathfrak g_f=(1)$, the followings hold.
\begin{enumerate}
\item[1.] If $v_{\mathfrak p}(2a_2)=0$, then
\[\overline{f(R)}^{\mathfrak p}=\{a_0+\alpha\}\sqcup\bigsqcup_{k=0}^\infty \bigsqcup_{i=1}^{N}(a_0+\alpha_{k,i}+\mathfrak p^{2k+1}),\]
where $\alpha,\alpha_{k,i}\in R$ and $N=\frac{\#\kappa-1}2$.
In particular, the set $\{\alpha_{0,1},\ldots,\alpha_{0,N}\}=S_{f,\mathfrak p}$.
\item[2.] If $v_{\mathfrak p}(a_2)\ge1$ then $\overline{f(R)}^{\mathfrak p}=R$.
\item[3.] If $v_{\mathfrak p}(2)\ge1$ and $v_{\mathfrak p}(a_2a_1)=0$, then
\[\overline{f(R)}^{\mathfrak p}=\bigsqcup_{i=1}^N(a_0+\alpha_{i}+\mathfrak p),\]
where $\alpha_{i}\in R$ and $N=\frac{\#\kappa}2$. 
\end{enumerate}
\end{thm}
\begin{rem}
If $v_{\mathfrak p}(2)\ge1$ and $v_{\mathfrak p}(a_1)\ge1$, then $f'(x)\mod \mathfrak p$ is identically $0$.
This makes this case much more difficult than other cases. We consider this after the proof of Theorem \ref{quad}. 
\end{rem}
\begin{proof}[Proof of Theorem \ref{quad}]
As the constant term $a_0$ only causes an additive translation, we can assume $a_0=0$ without loss of generality.
Let $f(x)=a_2x^2+a_1x\in R[x]$. 
\begin{enumerate}
    \item[1.] First, we deal with the case $v_{\mathfrak p}(2a_2)=0$. We note that ${\rm char}(R)\neq2$. In this case, $f'(x)\equiv 0\mod \mathfrak p$ has only one solution $x\equiv -\frac{a_1}{2a_2}\mod\mathfrak p$.
Corollary \ref{bit} leads \[\overline{f(R)}^{\mathfrak p}=\bigsqcup_{a\in S_{f,\mathfrak p}}(a+\mathfrak p)\sqcup E,\]
where $E$ is caused by the unique solution $x\equiv -\frac{a_1}{2a_2}\mod\mathfrak p$.
Let $\alpha_{\mathfrak p}$ be an element in $R$ such that $2a_2\mid(\alpha_{\mathfrak p} \pi-a_1)$ and $\alpha_{\mathfrak p}\not\in\mathfrak p$. Then $f'(\frac{\alpha_{\mathfrak p}\pi-a_1}{2a_2})\equiv0\mod\mathfrak p$ and \[f\left(\frac{\alpha_{\mathfrak p }\pi-a_1}{2a_2}\right)=\frac{(\alpha_{\mathfrak p}\pi-a_1)^2}{4a_2}+\frac{\alpha_{\mathfrak p}\pi a_1-a_1^2}{2a_2}=\frac{\alpha_{\mathfrak p}^2\pi^2-a_1^2}{4a_2}.\]
One may confirm that \[f(x)\equiv \frac{\alpha_{\mathfrak p}^2\pi^2-a_1^2}{4a_2}\mod \mathfrak p\Longleftrightarrow x\equiv\left(\frac{\alpha_{\mathfrak p}\pi-a_1}{2a_2}\right)\mod\mathfrak p.\]
Let $s=\pi^2\beta_1+\frac{\alpha_{\mathfrak p}^2\pi^2-a_1^2}{4a_2}\in\overline{f(R)}^{\mathfrak p}$. Then 
\begin{align*}
f\left(\pi \beta_2+\frac{\alpha_{\mathfrak p}\pi-a_1}{2a_2}\right)&=\pi^2\beta_1+\frac{\alpha_{\mathfrak p}^2\pi^2-a_1^2}{4a_2}\\
a_2\beta_2^2+\alpha_{\mathfrak p}\beta_2&=\beta_1.
\end{align*}
To characterize $\beta_1$, we consider the polynomial $f_1(x)=a_2x^2+\alpha_{\mathfrak p}x\in R[x]$ and $\overline{f_1(R)}^{\mathfrak p}$. Since $v_{\mathfrak p}(2a_2)=0$, we can apply the same argument with the above and express $\overline{f_1(R)}^{\mathfrak p}=\bigsqcup_{a\in S_{f_1,\mathfrak p}}(a+\mathfrak p)\sqcup E_1$
and \[E= \pi^{2}\left(\bigsqcup_{a\in S_{f_1,\mathfrak p}}(a+\mathfrak p)\sqcup E_1\right)+\frac{\alpha_{\mathfrak p}^2\pi^2-a_1^2}{4a_2}.\] 
Let $\alpha_{\mathfrak p}'\in R$  such that $2a_2\mid(\alpha_{\mathfrak p}' \pi-\alpha_{\mathfrak p})$ and $\alpha_p'\not\in\mathfrak p$. 
Applying the same argument as above, we choose $f_2(x)=a_2x^2+\alpha_{\mathfrak p}'x\in R[x]$ as the role of $f_1$ in $f$ and get other elements in $E$. 
Accordingly, we can characterize $E$ inductively.
In this procedure, there is unique element $x\in R$ such that $f'(x)\equiv f'_i(x)\equiv0\mod \mathfrak p$ for all $i$.
Therefore, we add this element to $E$.
Also, $\#S_{f,\mathfrak p}=\#S_{f_1,\mathfrak p}=\cdots=\frac{\#\kappa-1}2$ and we obtain the first statement.
\item[2.] Next, we assume $v_{\mathfrak p}(a_2)\ge1$. In this case, $f'(x)\equiv 0\mod\mathfrak p$ has no solution and if $f(x)\equiv f(y)\mod\mathfrak p$ then $(a_2(x+y)+a_1)(x-y)\equiv0\mod\mathfrak p$. Therefore, $f(x)\equiv f(y)\mod\mathfrak p$ if and only if $x\equiv y\mod\mathfrak p$ and for any $b\in \{0,c_2,\ldots,c_N\}$, there exists $a$ such that $f(a)\equiv b\mod\mathfrak p$. By Corollary \ref{bit}, it holds that $\overline{f(R)}^{\mathfrak p}=R$.
\item[3.]We assume $v_{\mathfrak p}(2)\ge1$ and $v_{\mathfrak p}(a_1a_2)=0$. In this case, we can also confirm that $f'(x)\equiv 0\mod\mathfrak p$ has no solution. Therefore, by Corollary \ref{bit} and the fact $\#S_{f,\mathfrak p}=\frac{\#\kappa}2=N$, we find that \[\overline{f(R)}^{\mathfrak p}=\bigsqcup_{i=1}^N(\alpha_{k,i}+\mathfrak p).\]
\end{enumerate}
This completes the proof of this theorem.
\end{proof}
In the following, we consider $\mathfrak p$-closure $\overline{f(R)}^{\mathfrak p}$ in the left case that $v_{\mathfrak p}(2)\ge1$ and $v_{\mathfrak p}(a_1)\ge1$. Shifting $x$, we can assume that $v_{\mathfrak p}(a_1)=v_{\mathfrak p}(2)$.
As $v_{\mathfrak p}(a_1)=v_{\mathfrak p}(2)\ge1$, we have $f'(x)\mod \mathfrak p$ is identically $0$. We consider elements of the form $\pi^{2v_{\mathfrak p}(a_1)}\beta_1$, which are caused by $f(0)\equiv0\mod\mathfrak p^{2v_{\mathfrak p}(a_1)}$. For $k\ge 2v_{\mathfrak p}(a_1)+1$
\begin{align*}
f(\pi^{v_{\mathfrak p}(a_1)} \beta_2)&\equiv \pi^{2v_{\mathfrak p}(a_1)}\beta_1\mod \mathfrak p^k\\
\pi^{2v_{\mathfrak p}(a_1)}a_2\beta_2^2+\pi^{v_{\mathfrak p}(a_1)} a_1\beta_2&\equiv \pi^{2v_{\mathfrak p}(a_1)}\beta_1\mod \mathfrak p^k\\
a_2\beta_2^2+\frac{a_1}{\pi^{v_{\mathfrak p}(a_1)}}\beta_2&\equiv \beta_1\mod \mathfrak p^{k-2v_{\mathfrak p}(a_1)}.
\end{align*}
Let $h(x)=a_2x^2+\frac{a_1}{\pi^{v_{\mathfrak p}(a_1)}}x$.
From the assumption, $v_{\mathfrak p}(\frac{a_1}{\pi^{v_{\mathfrak p}(a_1)}})=0$. Thus, we find \[\overline{h(R)}^{\mathfrak p}=\pi^{2v_{\mathfrak p}(a_1)}R\]
by Theorem \ref{quad}.
When we deal with other elements of the form $\pi^{2v_{\mathfrak p}(a_1)}\beta_1+a_2c_i^2+a_1c_i$ caused from $f(c_i)$, we need to estimate $v_{\mathfrak p}(2a_2c_i+a_1)$ for each $c_i\ (2\le i\le N)$ and use former discussion.
Therefore, we have \[\overline{f(R)}^{\mathfrak p}=\pi^{2v_{\mathfrak p}(a_1)}R\cup E,\]
where $E$ is caused by the elements of the form $\pi^{2v_{\mathfrak p}(a_1)}\beta_1+a_2c_i^2+a_1c_i$.
The fact that $f'(x)\mod \mathfrak p$ is identically $0$ makes dealing with $E$ very difficult.

\section{Bhargava factorial for integer coefficient quadratic polynomials}
\label{stirlingsec}
In Theorem \ref{quad}, we considered the $\mathfrak p$-adic closure $\overline{f(R)}^{\mathfrak p}$. However, it is difficult to characterize $\overline{f(R)}^{\mathfrak p}$ for all $\mathfrak p$. Therefore, we focus on the case $R=\mathbb Z$. 
In this situation, we replace prime ideals $p\mathbb Z$ with primes $p$ and use abbreviation $l!_{f(\mathbb Z)}$ for $l!_{f(\mathbb Z)}^{\mathbb Z}$. First, we present examples of $\overline{f(\mathbb Z)}^p$.
\begin{exam}
\label{example}
Let $s(x)=x^2$. For any odd prime $p$,
\[\overline{s(\mathbb Z)}^p=\{0\}\sqcup\bigsqcup_{k=0}^{\infty}\bigsqcup_{\substack{\left(\frac rp\right)=1\\1\le r\le p-1}}(rp^{2k}+p^{2k+1}\mathbb Z),\]
where $\left(\frac rp\right)$ is the Legendre symbol modulo $p$ defined by
\[\left(\frac rp\right)=\left\{\begin{array}{cl}
    1 &  \text{ if there exists $p\nmid a$ with $a^2\equiv r\mod p$},\\
     -1&\text{ if for any $a\in\mathbb Z$, $a^2\not\equiv r\mod p$},\\
     0 &\text{ if $r\equiv0\mod p$}.
\end{array}\right.\]
Also, if $p=2$, it holds that
\[\overline{s(\mathbb Z)}^2=\{0\}\sqcup\bigsqcup_{k=0}^{\infty}(2^{2k}+2^{2k+3}\mathbb Z).\]
\end{exam}
We first show the former case in the last example by using Theorem \ref{quad}.
\begin{proof}[Proof of the first assertion of Example \ref{example}]
For fixed odd prime $p$, the solution to the congruence $2x\equiv 0\mod p$ is $x\equiv 0\mod p$ and 
\[S_{s,p}=\left\{r\mod p~\left|~\left(\frac rp\right)=1\right\}\right..\]
By Theorem \ref{quad}, we have
\[\overline{s(\mathbb Z)}^p=\bigsqcup_{\substack{\left(\frac rp\right)=1\\1\le r\le p-1}}(r+p\mathbb Z)\sqcup E.\]
We next consider the case $p|s(x)$, which corresponds to $E$, that is, $x\equiv0\mod p$.
Letting $x=pk$, we find that $p^2|s(pk)=p^2k^2$.
It suffices to characterize the $p$-adic closure $\overline{s_1(\mathbb Z)}$ for $s_1(x)=\frac{s(pk)}{p^2}$. Since $s(x)=s_1(x)$, by the same argument as above, we find 
\[E=\overline{s_1(\mathbb Z)}^p=\bigsqcup_{\substack{\left(\frac rp\right)=1\\1\le r\le p-1}}(r+p\mathbb Z)\sqcup E_1,\]
and 
\[\overline{s(\mathbb Z)}^p=\bigsqcup_{k=0}^1\bigsqcup_{\substack{\left(\frac rp\right)=1\\1\le r\le p-1}}(rp^{2k}+p^{2k+1}\mathbb Z)\sqcup p^2E_1.\]
We can apply the same argument repeatedly with $p^2s_{i+1}(x)=s_i(px)$ and we note that $s(0)=s_i(0)=0$ for any $i\ge1$.
Therefore, we have
\[\overline{s(\mathbb Z)}^p=\{0\}\sqcup\bigsqcup_{k=0}^{\infty}\bigsqcup_{\substack{\left(\frac rp\right)=1\\1\le r\le p-1}}(rp^{2k}+p^{2k+1}\mathbb Z).\]
\end{proof}

The latter case of Example \ref{example} is obtained by the following lemma.
\begin{lem} 
\label{oddeven}
Let $f(x)=a_2x^2+a_1x+a_0\in\mathbb Z[x]$ with $\gcd(a_1,a_2)=1$ and $2|a_1$. Then
\[
\overline{f(\mathbb Z)}^2=\{a_0+\alpha\}\sqcup\bigsqcup_{k=0}^\infty (a_0+\alpha_{k}+2^{2k+3}\mathbb Z)\]
with $\alpha,\alpha_k\in\mathbb Z$.
\end{lem}
\begin{proof}
As in the proof of Theorem \ref{quad}, we can assume $a_0=0$.
Also, by the assumption, we find $2a_2+a_1\not \equiv a_1\mod 4$ and we can assume $a_1$ is a multiple of $4$ without loss of generality.
For any integer $k$, 
\[f(2k+1)=4\left(a_2k^2+\left(a_2+\frac{a_1}{2}\right)k\right)+a_2+a_1.\]
By the assumption, one can confirm that $a_2$ and $a_2+\frac{a_1}{2}$ are odd integers. By Theorem \ref{quad}, we have $\overline{f(2\mathbb Z+1)}^2=8\mathbb Z+(a_2+a_1)$. 

Next, we deal with $f(2\mathbb Z)$. We take an integer $\alpha_1$ such that $2a_2\alpha_1\equiv a_1\mod 8$, then
\[f(2k+\alpha_1)=4\left(a_2k^2+\left(a_2\alpha_1+\frac{a_1}{2}\right)k\right)+a_2\alpha_1^2+a_1\alpha_1.\]
As $a_2$ is an odd and $a_2\alpha_1+\frac{a_1}{2}$ is a multiple of $4$, we can apply the same argument as above and obtain the assertion,
\[
\overline{f(\mathbb Z)}^2=\{\alpha\}\sqcup(8\mathbb Z+a_2+a_1)\sqcup\bigsqcup_{k=1}^\infty (\alpha_{k}+2^{2k+3}\mathbb Z)\]
with $\alpha,\alpha_k\in\mathbb Z$,
inductively. This proves the lemma.
\end{proof}
\begin{thm}
\label{pord}
Let $f(x)=g(a_2x^2+a_1x)+a_0\in\mathbb Z[x]$ with $g\ge1$ and $\gcd(a_1,a_2)=1$. Then
\[l!_{f(\mathbb Z)}=g^l\prod_{p|a_2}p^{k(p,l)}\prod_{p\nmid a_2}p^{k(p,2l)}2^{-\delta_{\{2|a_1, 2\nmid a_2\}}},\]
where \[k(p,l)=\sum_{k=1}^\infty\left[\frac l{p^k}\right]\] and $\delta_{P}$ means that if $P$ is true then $\delta_P=1$, and otherwise $\delta_P=0$.
\end{thm}
\begin{proof}
It suffices to prove the case $g=1$.
First we consider $v_p(l!_{f(\mathbb Z)})$ for prime $p\mid a_2$. From Theorem \ref{quad}, we find $\overline{f(\mathbb Z)}^p=\mathbb Z$. Therefore, it holds that $v_p(l!_{f(\mathbb Z)})=v_p(l!)=k(p,l)$. 
Next, we deal with odd prime $p\nmid a_2$. It is known that $l!_{s(\mathbb Z)}=\frac{(2l)!}2$ for $s(x)=x^2$ \cite{bh00}. From Theorem \ref{quad}, we can calculate the $p$-adic closure of $f(\mathbb Z)$ as
\[
\overline{f(\mathbb Z)}^p=\{\alpha\}\sqcup\bigsqcup_{k=0}^\infty \bigsqcup_{i=1}^{\frac{p-1}2}(\alpha_{k,i}+p^{2k+1}\mathbb Z)\]
for odd prime $p$. As $\#\{f(a)\mod p ~|~f'(a)\not\equiv0\mod p\}=\#\{s(a)\mod p ~|~s'(a)\not\equiv0\mod p\}=\frac{p-1}2$, there exists a bijection $\phi:\overline{f(\mathbb Z)}^p\longrightarrow \overline{s(\mathbb Z)}^p$ such that $\alpha\longmapsto 0$ and $\alpha_{k,i}\longmapsto r_ip^{2k}$, where $\{r_1,\ldots,r_{\frac{p-1}2}\}$ is the set of all quadratic residues mod $p$. Therefore, $v_p(l!_{f(\mathbb Z)})=v_p(l!_{s(\mathbb Z)})$ and $v_p(l!_{f(\mathbb Z)})=v_p((2l)!)$ for any odd prime $p\nmid a_2$.

Finally, we consider $v_2(l!_{f(\mathbb Z)})$ under the assumption $a_2$ is odd. When $a_1$ is even, Lemma \ref{oddeven} gives 
\[
\overline{f(\mathbb Z)}^2=\{\alpha\}\sqcup\bigsqcup_{k=0}^\infty (\alpha_{k}+2^{2k+3}\mathbb Z)\]
with $\alpha,\alpha_k\in\mathbb Z$.
Then, by a bijection $\phi:\overline{f(\mathbb Z)}^2\longrightarrow \overline{s(\mathbb Z)}^2$ such that $\alpha\longmapsto 0$ and $\alpha_{k}\longmapsto 2^{2k}$, we find that $v_2(l!_{f(\mathbb Z)})=v_2(l!_{s(\mathbb Z)})=v_2((2l)!)-1$.  Next, we assume $a_1$ is odd, that is, $2\nmid a_2a_1$. Theorem \ref{quad} gives $\overline{f(\mathbb Z)}^2=2\mathbb Z$. It is known that the natural ordering $0,2,4,6,\ldots$ forms a $2$-ordering of $2\mathbb Z$ (see \cite{bh00}). Therefore, $v_2(l!_{f(\mathbb Z)})=v_2((2l)!)$.
We obtain the desired conclusion.
\end{proof}
If $|a_2|\le2$ then we can compute $l!_{f(\mathbb Z)}$ directly and obtain the following explicit formulas.
When $|a_1|=1$ then
\[l!_{f(\mathbb Z)}=\left\{\begin{array}{ll}
g^l(2l)!&2\nmid a_1,\\
\frac{g^l(2l)!}2& 2|a_1,
\end{array}\right.\]
and if $|a_2|=2$ then $l!_{f(\mathbb Z)}=g^l(2l)!2^{-l}$. In the followings, we present an estimate of the quotient $l!_{f(\mathbb Z)}$ by $g^l(2l!)$.
\begin{prop}
For any quadratic polynomial $f(x)=g(a_2x^2+a_1x)+a_0$ with $g\ge1$ and $\gcd(a_1,a_2)=1$, it holds that
\[\frac12\prod_{p|a_2} p^{-\frac{2l}p}\le\frac{l!_{f(\mathbb Z)}}{g^l (2l)!}\le \prod_{p|a_2} p^{-\frac{l}p+1}.\]
Moreover, we find that $\log l!_{f(\mathbb Z)}\sim 2l\log 2l$ by using the classical Stirling formula.
\end{prop}
\begin{proof}
By Theorem \ref{pord}, we find \[\frac{l!_{f(\mathbb Z)}}{g^l (2l)!}=\prod_{p|a_2} p^{k(p,l)-k(p,2l)}2^{-\delta_{\{2|a_1, 2\nmid a_2\}}}.\]
Since $\frac lp-1\le k(p,2l)-k(p,l)\le \frac{2l}p$ and $-\delta_{\{2|a_1, 2\nmid a_2\}}\in\{0,1\}$, it holds that
\[\frac12\prod_{p|a_2} p^{-\frac{2l}p}\le\frac{l!_{f(\mathbb Z)}}{g^l (2l)!}\le \prod_{p|a_2} p^{-\frac{l}p+1}.\]
\end{proof}
When $\deg f\ge3$, this problem becomes much difficult. As an example,
we consider the Bhargava factorial for $S(n)=f(\mathbb Z)$ with $f(x)=x^n$ and $n\ge3$. For any prime $p\nmid n$ we find that $f'(x)\equiv0\mod p$ if and only if $p|x$. Since $f(x)=x^n$ is an endomorphism of $(\mathbb Z/p\mathbb Z)^\times$, we have $\#\Ima f=\#S_{f,p}=\frac{p-1}{g}$, where $g=\gcd(p-1,n)$. From Corollary \ref{bit}, we find
\[\overline{S(n)}^p\supset\bigsqcup_{a\in S_{f,p}}a+p\mathbb Z.\]
By computing the $p$-factor of the Bhargava factorial directly, we obtain the following estimate
\begin{equation}
\label{estimate} 
v_p(l!_{S(n)})<\sum_{k=1}^\infty\frac{gl}{p^{k-1}(p-1)}.    
\end{equation}
Also, Fares and Johnson \cite{fj12} proved that for the primes $p$ with $n|(p-1)$, 
\begin{equation}
\label{estimate2} 
v_p(l!_{S(n)})=v_p((nl)!).
\end{equation}
Let $P_n$ be the set of all primes $p$ with $n|(p-1)$ and let $Q_n$ be the set of all primes $p$ with $\gcd(p-1,n)\le2$.
By inequality (\ref{estimate}) and identity (\ref{estimate2}), we can observe the difference of the behavior of $v_p(l!_{S(n)})$ between primes in $P_n$ and these in $Q_n$. As both $P_n$ and $Q_n$ are infinite sets, it is difficult to give analytic estimates for $l!_{S(n)}$. 
For example, identity (\ref{estimate2}) leads that $v_p(l!_{S(3)})=v_p((3l)!)$ for any prime $p\in P_3$.
On the other hand, for any odd prime $p\in Q_3$, we calculate
\[
v_p(l!_{S(3)})<\sum_{k=1}^\infty\frac{l}{p^{k-1}(p-1)}\nonumber\le\sum_{k=1}^\infty\frac{2l}{p^{k}}=v_p((2l)!).
\]
Both $v_p(l!_{S(3)})=v_p((3l)!)$ and $v_p(l!_{S(3)})<v_p((2l)!)$ hold at infinitely many primes. We note that a polynomial $f(x)$ is not an endomorphism of $(\mathbb Z/p\mathbb Z)^\times$ in general. Therefore, when $\deg f\ge3$, this is much more difficult to give analytic estimates for $l!_{f(\mathbb Z)}$ than the former case.
\section*{Acknowledgement}
This work was supported by Grant-in-Aid for JSPS Research Fellow (Grant Number: JP19J10705).


\begin{thebibliography}{99}
\bibitem{ac15} Adam, D., Chabert, J. L. (2015). On Bhargava's factorials of the set of twin primes in $\mathbb Z$ and in $\mathbb F_q [T]$. International Journal of Number Theory {\bf 11}, no.6, 1941--1959.
\bibitem{bh97} Bhargava, M. (1997). P-orderings and polynomial functions on arbitrary subsets of Dedekind rings. J. Reine Angew. Math. {\bf 490}, 101--127. 
\bibitem{bh98} Bhargava, M. (1998). Generalized factorials and fixed divisors over subsets of a Dedekind domain, J. Number Theory {\bf 72}, 67--75.
\bibitem{bh00} Bhargava, M. (2000). The factorial function and generalizations. Amer. Math. Monthly 107, no. {\bf 9}, 783--799. 
\bibitem{Da20} Diaz, B. (2020). Asymptotics on a class of Legendre formulas. Preprint, arXiv:2010.13645.
\bibitem{ev12}  Evrard, S. (2012). Bhargava's factorials in several variables. J. Algebra 372, 134--148.
\bibitem{fj12} Fares, Y., Johnson, K. (2012). The Characteristic Sequence and p-Orderings of the Set of d-th Powers of Integers. Integers {\bf 12}, no. 5, 865--876.
\bibitem{ma} H. Matsumura. (1989). Commutative ring theory (No. 8). Cambridge university press.
\bibitem{kad18} Rajkumar, K., Reddy, A. S., Semwal, D. P. (2018). Fixed divisor of a multivariate polynomial and generalized factorial in several variables. J. Korean Math. Soc. 55, no. 6, 1305--1320. 
\end{thebibliography}
\end{document}